\documentclass[11pt,final]{article}

\usepackage{amsmath}
\usepackage{amsthm}
\usepackage{amssymb}
\usepackage{color}

\numberwithin{equation}{section}

\theoremstyle{definition}\newtheorem{definition}{Definition}
\theoremstyle{plain}\newtheorem{theorem}[definition]{Theorem}
\theoremstyle{plain}\newtheorem{proposition}[definition]{Proposition}
\theoremstyle{plain}\newtheorem{corollary}[definition]{Corollary}
\theoremstyle{plain}\newtheorem{lemma}[definition]{Lemma}
\theoremstyle{definition}
\theoremstyle{definition}
\theoremstyle{definition}
\theoremstyle{definition}\newtheorem{remark}[definition]{Remark}

\newcommand{\cR}{{\mathcal R}}
\newcommand{\cN}{{\mathcal N}}
\newcommand{\bbN}{\mathbb{N}}

\begin{document}

\title{Duality of uncomplemented null-space and non-closed range and its impact on the regularization of Mazur-type operators}

\author{
\textsc{Jens Flemming}%
\footnote{HTWD -- University of Applied Sciences, Faculty of Informatics/Mathematics, 01069 Dresden, Germany, e-mail: jens.flemming@htw-dresden.de}
~and
\textsc{Bernd Hofmann}%
\footnote{Chemnitz University of Technology, Faculty of Mathematics, 09107 Chemnitz, Germany, e-mail: bernd.hofmann@mathematik.tu-chemnitz.de}
}

\date{\today
}

\maketitle

\begin{abstract}
Up to now operators with uncomplemented null-space have been playing only a minor role in regularization theory for ill-posed inverse problems in Banach spaces. We show that ill-posedness due to an uncomplemented null-space and ill-posedness due to a non-closed range can be regarded as dual concepts and we may switch between both views at will by restricting or extending the ill-posed operator under consideration in a suitable well-defined way. We apply the duality result to Mazur-type operators, which are a class of bounded linear operators with uncomplemented null-space. Mazur-type operators in their natural form are not accessible to Tikhonov-type regularization, but after transfering ill-posedness to the range Tikhonov regularization can be applied. On the other hand, the null-space view of ill-posedness may open up an alternative path for developing new regularization methods for operators ill-posed in the classical sense, that is, for operators ill-posed due to a non-closed range.
\end{abstract}

\section{Motivation and preliminaries}

We are concerned with linear operator equations
\begin{equation}
A\,x\,=\,y\,,
\end{equation}
where $A:X\to Y$ is a bounded linear operator mapping between Banach space $X$ and $Y$. Here, $y\in Y$ is given and $x\in X$ is being sought. Of particular interest are operators $A$ lacking a continuous (generalized) inverse. To solve such \emph{ill-posed} operator equations some kind of regularization technique has to be applied, cf., e.g.,~\cite{SchKalHofKaz12} and \cite{Scherzetal09}. The following definition plays an import role in discussing invertibility of $A$.

\begin{definition}\label{def:complemented}
A closed subspace $U\subseteq X$ of a Banach space $X$ is \emph{complemented} if there exists another closed subspace $V\subseteq X$ such that
\begin{equation}\label{eq:complemented}
U\cap V=\{0\}\quad\text{and}\quad U+V=X.
\end{equation}
Otherwise, $U$ is \emph{uncomplemented}.
\end{definition}

Note that for each closed subspace $U\subseteq X$ there is always another (possibly non-closed) subspace $V\subseteq X$ such that \eqref{eq:complemented} holds true. The difficulty lies in finding a closed subspace $V$. The Lindenstrauss-Tzafiri Theorem \cite{LinTza71} states that a Banach space contains uncomplemented subspaces if and only if it is not isomorphic to a Hilbert space.

If $A$ is injective, then the inverse $A^{-1}:\cR(A)\to X$ is well-defined on the range $\cR(A)$ of $A$ and the inverse is bounded if and only if $\cR(A)$ is a closed subset of $Y$. In case of closed range, the inverse may be extended to a bounded linear \emph{generalized inverse} $A^\dagger:Y\to X$ if the range $\cR(A)$ is complemented. But we do not consider such generalizations here, instead we refer the reader to the original work \cite{NasVot76}. Note that the closure of $\cR(A)$ is a Banach space on its own. So without loss of generality we may assume that either $\cR(A)$ is dense in $Y$ or $\cR(A)=Y$.

If $A$ is not injective, that is $A$ has a non-trivial null-space $\cN(A)\neq\{0\}$, we may take some (possibly non-closed) subspace $U$ such that \eqref{eq:complemented} holds with $V:=\cN(A)$. Then the restriction $A|_U:U\to Y$ is injective and has range $\cR(A|_U)=\cR(A)$. Thus, the discussion above applies to $A|_U$. But it turns out that even for a closed range $\cR(A)$ the corresponding inverse $(A|_U)^{-1}$ may be unbounded. In \cite[Proposition~1.11]{Flemmingbuch18} it has been shown that
$(A|_U)^{-1}$ is bounded if and only if $U$ is a complement of $\cN(A)$, that is, if and only of $U$ is closed. Consequently, non-closed range $\cR(A)$ and uncomplemented null-space $\cN(A)$ are two sources of ill-posedness that appear independent of one another at first glance.

Operators with uncomplemented null-space have been playing almost no role in the theory of ill-posed inverse problems during past decades. Results on ill-posedness phenomena and suitable regularization techniques concentrated on the situation of non-closed range, although uncomplemented null-spaces (and also uncomplemented ranges) have already been mentioned in \cite{NasVot76,Nas87}. There, the possible presence of uncomplemented subspaces has been incorporated into the definition of ill-posedness. That definition is discussed and slightly modified in the recent articles \cite{HofKin26,FleHof26} in the light of uncomplemented null-spaces and and strictly singular operators.

In the present article's main section, Section~\ref{sec:duality}, we show that ill-posedness in the sense of lacking boundedness of (generalized) inverses resulting from a non-closed range is essentially the same as ill-posedness resulting from an uncomplemented null-space. Hence, both sources of ill-posedness are anything but independent, however they can be regarded as concepts dual to each other. Operators with uncomplemented null-space may be transformed into operators with non-closed range and vice verse.

On the one hand, our results provide a path for applying well-known regularization methods to operators with uncomplemented null-space (by transfering ill-posedness to the range). We illustrate this approach in Section~\ref{sec:mazur} with Mazur-type operators. On the other hand, our results may motivate new regularization methods by transfering ill-posedness to the null-space and developing regularization techniques for such situations.

\section{Duality of uncomplemented null-space and non-closed range}\label{sec:duality}

Following definition was introduced in \cite{Mur45}.

\begin{definition}\label{def:quasi}
A closed subspace $U\subseteq X$ of a Banach space $X$ is \emph{quasi-complemented} if there exists a closed subspace $V\subseteq X$ such that
\begin{equation}\label{eq:quasi}
U\cap V=\{0\}\quad\text{and}\quad U+V\text{ is dense in }X.
\end{equation}
\end{definition}

In contrast to complementedness, quasi-complementedness is a rather weak property. In \cite{Mac46} it was shown that in a separable Banach space each closed subspace is quasi-complemented.

The following theorem holds, in particular, for operators $B$ with uncomplemented null-space. Note that assumed quasi-complementedness of the null-space $\cN(B)$ is guaranteed if $X$ is separable.

\begin{theorem}\label{th:restriction}
Let $B:X\to Y$ be a bounded operator between Banach spaces $X$ and $Y$ and let $U\subseteq X$ be a quasi-complement of the null-space $\cN(B)$.
Then the restriction $A:=B|_U$ is injective and the range $\cR(A)$ of $A$ is dense in the range $\cR(B)$ of $B$.
\end{theorem}

\begin{proof}
We have $\cN(B)\cap U=\{0\}$ and $\cN(B)+U$ is dense in $X$. The former implies that $A$ is injective. The latter, for each $x\in X$ yields sequences $(z^{(k)})_{k\in\bbN}$ in $\cN(B)$ and $(u^{(k)})_{k\in\bbN}$ in $U$ with $z^{(k)}+u^{(k)}\to x$ for $k\to\infty$. From $B\,(z^{(k)}+u^{(k)})=B\,u^{(k)}=A\,u^{(k)}$ we see $A\,u^{(k)}\to B\,x$ if $k\to\infty$, that is, the range of $A$ is dense in the range of $B$.
\end{proof}

\begin{corollary}\label{th:cor_restriction}
Let $X$ be separable. Then for each surjective bounded linear operator $B:X\to Y$ there is a closed subspace $U\subseteq X$, such that the restriction $A:=B|_U$ is injective and the range $\cR(A)$ is dense in $Y$.
\end{corollary}

\begin{proof}
Separability of $X$ yields a quasi-complement $U$ of $\cN(B)$. Thus, Theorem~\ref{th:restriction} applies.
\end{proof}

If $B$ in the corollary is ill-posed due to an uncomplemented null-space, then the corollary yields an, up to some extent, equivalent operator $A$ with possibly non-closed range. Thus, ill-posedness due to an uncomplemented null-space is transfered to ill-posedness due to a non-closed range. That $A$ contains all essential information from $B$ can be seen from the fact that $\cN(B)+U$ is dense in $X$, that is, we only `loose' the null-space of $B$.

\begin{proposition}\label{th:restriction_compl}
Let $B$, $U$, $A$ be as in Corollary~\ref{th:cor_restriction}. The range $\cR(A)$ is closed if and only if $U$ is a complement of the null-space $\cN(B)$.
\end{proposition}

\begin{proof}
If $\cR(A)$ is closed, that is $\cR(A)=Y$, take $u\in U$ such that $A\,u=B\,x$ for some arbitrary $x\in X$. Then $x=(x-u)+u$ shows that $U$ is a complement of $\cN(B)$.
\par
If, on the other hand, $U$ is a complement of $\cN(B)$, we find $z\in\cN(B)$ and $u\in U$ with $x=z+u$ for some arbitrary $x\in X$. Then $A\,u=B\,(z+u)=B\,x$ shows that $\cR(A)=\cR(B)=Y$.
\end{proof}

In other words, the proposition states that in Corollary~\ref{th:cor_restriction} ill-posedness due to an uncomplemented null-space occurs for $B$ if and only if ill-posedness due to a non-closed range occurs for $A$.

\begin{remark}\label{rm:compl}
From Proposition~\ref{th:restriction_compl} we cannot deduce that a complemented $\cN(B)$ implies a closed $\cR(A)$ for all quasi-complements $U$ of $\cN(B)$. We only get existence of at least one complement $U$ of $\cN(B)$ such that the corresponding operator $A$ has closed range. For other quasi-complements, $\cR(A)$ may be non-closed although $\cN(B)$ is complemented.
The other way round, if we know that $\cR(A)$ is non-closed for a specific choice of $U$, we cannot deduce that $\cN(B)$ is uncomplemented. Perhaps there is another subspace $U$ that is a complement of $\cN(B)$ (and corresponding $A$ has closed range).
Note that even in Hilbert spaces closed subspaces may have quasi-complements which aren't complements (see \cite[Section~I.15]{Hal51} for an example).
\end{remark}

The next theorem is the converse result to Corollary~\ref{th:cor_restriction}. The assumptions could be given in simpler form but we keep them quite general to maintain a high degree of symmetry to Corollary~\ref{th:cor_restriction}. For instance, the assumptions are satisfied if $W:=Y$, $C$ is the identity mapping in $Y$, and $\tilde{X}:=Y\times X$ with norm $\|(w,u)\|_{\tilde{X}}:=\|w\|_W+\|x\|_X$.
The proof of the theorem relies on the Bartle-Graves Theorem \cite[Cor.\ 5J.4]{DonRoc14} stated in the following Lemma.

\begin{lemma}\label{th:bartle_graves}
Let $B:\tilde{X}\to Y$ be a surjective bounded linear operator between arbitrary Banach spaces $\tilde{X}$ and $Y$. Then there exists a continuous (and possibly non-linear) mapping $B^\dagger:Y\to\tilde{X}$ such that $B\,B^\dagger(y)=y$ for all $y\in Y$.
\end{lemma}

\begin{proof}
See \cite[Cor.\ 5J.4]{DonRoc14}.
\end{proof}

\begin{theorem}\label{th:extension}
Let $A:X\to Y$ be an injective bounded linear operator between Banach spaces $X$ and $Y$ with dense range $\cR(A)$. Further let $\tilde{X}=W+U$ be the direct sum of a closed subspace $W$, for which there exists a surjective bounded linear operator $C:W\to Y$ and a closed subspace $U$ isomorphic to $X$ with isomorphism $T:X\to U$. Then there exists a surjective bounded linear operator $B:\tilde{X}\to Y$ such that $A=B\,T$ and $U$ is a quasi-complement of $\cN(B)$ in $\tilde{X}$.
\end{theorem}

\begin{proof}
For $w+u\in\tilde{X}$ with $w\in W$ and $u\in U$ set
\begin{equation}
B\,(w+u):=C\,w+A\,T^{-1}\,u.
\end{equation}
Then $B:\tilde{X}\to Y$ is a well-defined surjective linear operator and $A=B\,T$. For proving boundedness of $B$, observe that $W$ and $U$ are a pair of complemented subspaces. Thus, the projection mappings $w+u\mapsto w$ and $w+u\mapsto u$ are continuous (cf.\ \cite[Theorem~3.2.14]{Meg98}) and we obtain
\begin{align}
\|B\,(w+u)\|_{Y}
&\leq\|C\|\,\|w\|_{\tilde{X}}+\|A\,T^{-1}\|\,\|u\|_{\tilde{X}}\\
&\leq\max\bigl\{c_1\,\|C\|,\;c_2\,\|A\,T^{-1}\|\bigr\}\,\|w+u\|_{\tilde{X}}
\end{align}
with constants $c_1>0$, $c_2>0$.
\par
Injectivity of $A$ ensures $\cN(B)\cap U=\{0\}$. To show denseness of $\cN(B)+U$ take $w+u\in\tilde{X}$ with $w\in W$ and $u\in U$. By the denseness of $\cR(A)$ in $Y$ there exists a sequence $(u^{(m)})_{m\in\bbN}$ in $U$ with $B\,u^{(m)}\to B\,(w+u)$ if $m\to\infty$. Set
\begin{equation}
\tilde{x}^{(m)}:=w+u-B^\dagger(B\,(w+u))-\bigl(u^{(m)}-B^\dagger(B\,u^{(m)})\bigr),
\end{equation}
where $B^\dagger:Y\rightarrow\tilde{X}$ is the continuous mapping from \ref{th:bartle_graves}. Obviously, $\tilde{x}^{(m)}\in\cN(B)$. Further,
\begin{equation}
\tilde{x}^{(m)}+u^{(m)}-(w+u)=B^\dagger(B\,u^{(m)})-B^\dagger(B\,(w+u))\to 0
\end{equation}
if $m\to\infty$ by continuity of $B^\dagger$, that is, $\tilde{x}^{(m)}+u^{(m)}\to w+u$.
\end{proof}

If $A$ in the theorem is ill-posed due to a non-closed range $\cR(A)\neq Y$, the theorem yields an extension $B$ with closed range and possibly uncomplemented null-space. Thus, ill-posedness of $A$ due to a non-closed range is transfered to ill-posedness of $B$ due to an uncomplemented null-space. That $B$ does not `add' more than necessary can be seen from the fact that $X$ is isomorphic to a quasi-complement $U$ of $\cN(B)$. This also makes clear that the choice of $C$ in the theorem does not matter. $C$ is an arbitrary surjective operator without any connection to $A$ and only serves as a tool for `filling the gap' between $\cR(A)$ and $Y$ by enlarging the null-space.

\begin{proposition}\label{th:extension_compl}
Let $A$, $B$, $C$, $U$ be as in Theorem~\ref{th:extension}.
The range $\cR(A)$ is closed if and only if $U$ is a complement of the null-space $\cN(B)$.
Further, $\cN(B)$ is complemented in $\tilde{X}$ if $\cN(C)$ is complemented in $W$.
\end{proposition}

\begin{proof}
We may apply Corollary~\ref{th:cor_restriction} to $B$ and $U$ to obtain an operator $\tilde{A}=B|_U:U\to Y$. By Proposition~\ref{th:restriction_compl} its range $\cR(\tilde{A})$ is closed if and only if $U$ is a complement of $\cN(B)$. From $A=\tilde{A}\,T$ with the isomorphism $T$ from Theorem~\ref{th:extension} we see that $\cR(A)$ is closed if and only if $\cR(\tilde{A})$ is closed.
\par
If $\cN(C)$ is complemented in $W$, we find a closed subspace $V$ in $W$ such that $\cN(C)\cap V=\{0\}$ and $\cN(C)+V=W$. We show that $V$ also is a complement of $\cN(B)$.
Let $\tilde{x}\in\cN(B)\cap V$. Then $\tilde{x}\in W$ and $C\,\tilde{x}=B\,\tilde{x}=0$, that is, $\tilde{x}\in\cN(C)\cap V$, showing $\tilde{x}=0$.
Now take arbitrary $\tilde{x}\in\tilde{X}$. Then there are $w\in W$ and $u\in U$ such that $\tilde{x}=w+u$. Further, there are $z\in\cN(C)$ and $v\in V$ such that $w=z+v$. Surjectivity of $C$ guarantees existence of some $\tilde{w}\in W$ with $C\,\tilde{w}=-A\,u$, that is, with $\tilde{w}+u\in\cN(B)$. For $\tilde{w}$ we find $\tilde{z}\in\cN(C)$ and $\tilde{v}\in V$ such that $\tilde{w}=\tilde{z}+\tilde{v}$. Thus, we have
\begin{equation}
\tilde{x}=z+v+u=\tilde{w}+u+z-\tilde{z}+v-\tilde{v}
\end{equation}
with $\tilde{w}+u+z-\tilde{z}\in\cN(B)$ (note that $\cN(C)\subseteq\cN(B)$) and $v-\tilde{v}\in V$. This shows $\tilde{X}=\cN(B)+V$.
\end{proof}

Analogously to Proposition~\ref{th:restriction_compl}, Proposition~\ref{th:extension_compl} states that in Theorem~\ref{th:extension} ill-posedness due to an uncomplemented null-space occurs for $B$ if and only if ill-posedness due to a non-closed range occurs for $A$.

Note that the Proposition~\ref{th:extension_compl} does not state, that $\cN(C)$ is complemented if $\cN(B)$ is complemented. The authors believe that uncomplemented $\cN(C)$ is possible even if $\cN(B)$ is complemented. But a rigorous proof is missing.

\section{Regularization of Mazur-type operators}\label{sec:mazur}

In \cite[page 111, item (e)]{BanMaz33} Banach and Mazur introduced a class of bounded linear operators mapping the sequence space $\ell^1$ onto a separable Banach space $Y$.

\begin{definition}\label{def:mazur}
A bounded linear operator $B:\ell^1\to Y$ mapping into a separable Banach space $Y$ is of \emph{Mazur-type} if there is a sequence $(\zeta^{(k)})_{k\in\bbN}$ in $Y$ which is dense in the unit sphere and satisfies
\begin{equation}
B\,x=\sum_{k=1}^\infty x_k\,\zeta^{(k)}\quad\text{for all }x=(x_1,x_2,\ldots)\in\ell^1.
\end{equation}
\end{definition}

Mazur-type operators obviously are well-defined and bounded linear operators. Note that each separable Banach space $Y$ and each choice of the sequence $(\zeta^{(k)})_{k\in\bbN}$, for fixed $Y$, yield different Mazur-type operators.

Mazur-type operators are surjective and their null-spaces are almost always uncomplemented.

\begin{proposition}
Let $B:\ell^1\to Y$ be a Mazur-type operator with defining sequence $(\zeta^{(k)})_{k\in\bbN}$. Then $B$ is surjective and the null-space $\cN(B)$ is complemented if and only if $Y$ is isomorphic to $\ell^1$.
\end{proposition}

\begin{proof}
Surjectivity is shown in \cite[proof of Theorem 2.3.1]{AlbKal06}. That $\cN(B)$ is uncomplemented if $Y$ is not isomorphic to $\ell^1$ is a consequence of \cite[Corollary 2.3.3]{AlbKal06}. It remains to show that $\cN(B)$ is complemented if $Y$ is isomorphic to $\ell^1$.
\par
Let $Y$ be isomorphic to $\ell^1$ and denote by $\tilde{S}$ an isomorphism from $\ell^1$ onto $Y$. Moreover, set $f^{(l)}:=\tilde{S}\,e^{(l)}$ for $l\in\bbN$, where $e^{(l)}$ has a one at position $l$ and zeros else. Then $(f^{(l)})_{l\in\bbN}$ is a basis in $Y$ equivalent to $(e^{(l)})_{l\in\bbN}$ in $\ell^1$ (cf.\ \cite[Prop.~4.3.2]{Meg98}). In particular, we have
\begin{equation}
Y=\left\{\sum_{l=1}^\infty a_l\,f^{(l)}:\;a\in\ell^1\right\}
\end{equation}
(cf.\ \cite[Def.~4.3.1]{Meg98}). Without loss of generality we may assume that $(f^{(l)})_{l\in\bbN}$ is normalized, because there are constants $c>0$ and $C>0$ such that
\begin{equation}
c=c\|e^{(l)}\|_{\ell^1}\leq\|f^{(l)}\|_Y\leq C\|e^{(l)}\|_{\ell^1}=C\quad\text{for }l\in\bbN.
\end{equation}
Normalization of the basis yields
\begin{equation}
\sum_{l=1}^\infty a_l\,f^{(l)}=\sum_{l=1}^\infty a_l\,\|f^{(l)}\|_Y\,\frac{f^{(l)}}{\|f^{(l)}\|_Y}
\end{equation}
for coefficients $a_1,a_2,\ldots$. Thus, from \cite[Theorem~4.3.6]{Meg98} we see that the normalized basis is equivalent to $(e^{(l)})_{l\in\bbN}$, too.
\par
By \cite[Prop.~4.3.4, Prop.~4.3.2]{Meg98} we find a sequence $(k_l)_{l\in\bbN}$ in $\bbN$ and an isomorphism $S:Y\to Y$ such that $(\zeta^{(k_l)})_{l\in\bbN}$ is a Schauder basis in $Y$ and $\zeta^{(k_l)}=S\,f^{(l)}$ for $l\in\bbN$. Thus, $(\zeta^{(k_l)})_{l\in\bbN}$ is equivalent to $(e^{(l)})_{l\in\bbN}$, too, resulting in
\begin{equation}
Y=\left\{\sum_{l=1}^\infty a_l\,\zeta^{(k_l)}:\;a\in\ell^1\right\}.
\end{equation}
Set
\begin{equation}
U:=\overline{\mathrm{span}\,\{e^{(k_l)}:\;l\in\bbN\}}^{\ell^1}.
\end{equation}
Obviously, $U$ is complemented and isomorphic to $\ell^1$ and the restriction $A:=B|_U$ is surjective. For $x\in\ell^1$ we find $u\in U$ with $A\,u=B\,x$, that is, $x-u\in\cN(B)$. The resulting decomposition $x=x-u+u$ shows $\ell^1=\cN(B)+U$. Injectivity of $A$ (cf.\ proof of Proposition~\ref{th:restriction}) implies $\cN(B)\cap U=\{0\}$. This indicates that the subspace $U$ is a complement of $\cN(B)$ in $\ell^1$, which completes the proof of the proposition.
\end{proof}

Due to their uncomplemented null-space Mazur-type operators are ill-posed and regularization techniques have to be employed for obtaining stable (approximate) solutions to operator equations $B\,x=y$. In \cite{FleHof26} it was shown that Mazur-type operators are not accessible to Tikhonov regularization
\begin{equation}
\|B\,x-y\|_Y^p+\alpha\,\|x\|_{\ell^1}\to\min_{x\in\ell^1}
\end{equation}
with $p\geq 1$ and regularization parameter $\alpha>0$, because Mazur-type operators are not weak*-to-weak continuous. In particular, it was shown there (cf.\ \cite[Theorem~3.3]{FleHof26}) that Tikhonov regularized solutions exist only for a dense subset of right-hand sides $y\in Y$.

By applying Theorem~\ref{th:restriction} to a Mazur-type operator $B$, we obtain an injective bounded linear operator $A$ with non-closed range, which in principle is accessible to Tikhonov regularization. The next proposition shows that indeed such $A$ is weak*-to-weak continuous. Moreover, the quasi-complement $U$ of the null-space $\cN(B)$ required in Theorem~\ref{th:restriction} can be made explicit in terms of the defining sequence $(\zeta^{(k)})_{k\in\bbN}$ of the Mazur-type operator under consideration. Weak*-to-weak continuity will be based on the existence of a \emph{shrinking} basis in $Y$, see, e.g., \cite[Def.~3.2.5]{AlbKal06}. Each reflexive Banach space has a shrinking basis, \cite[Theorem~3.2.13]{AlbKal06}.
For separable Hilbert spaces $Y$, the proposition was already proven in \cite[Proposition~4]{HofFle26} .

\begin{proposition}\label{th:restriction_mazur}
For each Mazur-type operator $B:\ell^1\to Y$ mapping onto a separable Banach space $Y$ there is a closed and complemented subspace $U\subseteq\ell^1$ isomorphic to $\ell^1$ such that $A:=B|_U$ is injective, the range of $A$ is dense in $Y$, and $U$ is a quasi-complement of $\cN(B)$. If $Y$ has a shrinking basis, then $A$ is weak*-to-weak continuous w.\,r.\,t.\ the predual $c_0$ of $\ell^1$.
\end{proposition}

\begin{proof}
Let $(\zeta^{(k)})_{k\in\bbN}$ be the sequence defining $B$ and take a normalized Schauder basis $(f^{(k)})_{k\in\bbN}$ in $Y$. Then by \cite[Prop.~4.3.4, Prop.~4.3.2]{Meg98} we find a sequence $(k_l)_{l\in\bbN}$ in $\bbN$ and an isomorphism $S:Y\to Y$ such that $(\zeta^{(k_l)})_{l\in\bbN}$ is a Schauder basis in $Y$ and $\zeta^{(k_l)}=S\,f^{(l)}$ for $l\in\bbN$.
\par
Set
\begin{equation}
U:=\overline{\mathrm{span}\,\{e^{(k_l)}:\;l\in\bbN\}}^{\ell^1}.
\end{equation}
Obviously, $U$ is complemented and isomorphic to $\ell^1$. Further, $A:=B|_U$ is injective: for $u=\sum_{l=1}^\infty u_l\,e^{(k_l)}\in U$ with $A\,u=0$ we have
\begin{equation}
0=B\,u
=\sum_{l=1}^\infty u_l\,\zeta^{(k_l)}
=\sum_{l=1}^\infty u_l\,S\,f^{(l)}
=S\left(\sum_{l=1}^\infty u_l\,f^{(l)}\right)
\end{equation}
and, thus, $\sum_{l=1}^\infty u_l\,f^{(l)}=0$, which implies $u=0$.
To show that $\cR(A)$ is dense in $Y$ take $y\in Y$ and decompose it w.\,r.\,t.\ the basis $(\zeta^{(k_l)})_{l\in\bbN}$:
\begin{equation}
y=\sum_{l=1}^\infty a_l\,\zeta^{(k_l)}
\end{equation}
with some sequence of scalars $a_1,a_2,\ldots$. Set
\begin{equation}
u^{(m)}=\sum_{l=1}^m a_l\,e^{(k_l)}\in U.
\end{equation}
Then $A\,u^{(m)}=\sum_{l=1}^m a_l\,\zeta^{(k_l)}\to y$ if $m\to\infty$.
\par
$\cN(B)\cap U=\{0\}$ and density of $\cN(B)+U$ in $\ell^1$ can be proven in the same way as in the proof of Theorem~\ref{th:extension}.
\par
To obtain weak*-to-weak continuity, due to \cite[Lemma~9.3]{Flemmingbuch18} it suffices to show that $B\,e^{(k_l)}$ converges weakly to zero for $l\to\infty$. We have $B\,e^{(k_l)}=\zeta^{(k_l)}=S\,f^{(l)}$ and by assumption $(f^{(l)})_{l\in\bbN}$ is a shrinking basis. Thus, by \cite[Prop.~3.2.7]{AlbKal06} $(f^{(l)})_{l\in\bbN}$ converges weakly to zero in $Y$. Because bounded linear operators are always weak-to-weak continuous, this proves the assertion.
\par
Note that the construction of $U$ is based on the fact that each Banach space has a normalized basis $(f^{(l)})_{l\in\bbN}$. But in the proposition we require existence of a shrinking basis without being normalized. In fact, normalizing a shrinking basis yields a shrinking basis again. This can be seen, for instance, from \cite[Prop.~4.4.7, Prop.~4.1.5]{Meg98}.
\end{proof}

Note that Mazur-type operators are not only an example to which Corollary~\ref{th:cor_restriction} can be applied. But they may also serve as a canonical surjection $C$ in Theorem~\ref{th:extension} if $W=\ell^1$. Choosing $W\neq Y$ provides some control over the properties of the extended domain $\tilde{X}$ of the operator $B$. If $Y$ is not isomorphic to $\ell^1$ this choice of $W$ always yields an uncomplemented null-space $\cN(B)$ by Proposition~\ref{th:extension_compl}.

\section{Discussion}

Corollary~\ref{th:cor_restriction} and Theorem~\ref{th:extension} show that there is a tight one-to-one correspondence between injective operators $A$ with dense range und surjective operators $B$. In both cases, $A$ is a restriction of $B$ and $B$ is an extension of $A$ (up to the isomorphism $T$). The operator $A$ can be regarded as the \emph{essential part} of $B$ in the sense that the domain of $A$ is a quasi-complement of the null-space of $B$.

With the results presented in this article, one may transfer ill-posedness due to a non-closed range to ill-posedness expressed in terms of a relationship between the null-space and one of its quasi-complements, and vice versa. See Remark~\ref{rm:compl} for a more detailed interpretation of this duality.

Whether characterizing ill-posedness in terms of properties of the null-space offers added value depends on whether future ideas emerge on how to algorithmically overcome ill-posedness caused by a uncomplemented null space.
 The Bartle-Graves Theorem formulated in Lemma~\ref{th:bartle_graves} always yields a (possibly nonlinear) continuous inverse, even if there is no bounded linear inverse. Attempting to algorithmically approximate such a Bartle-Graves inverse could prove to be a promising path and a rich source of new regularization methods. But at the moment for the authors it is not clear how to tackle this task.

Transfering ill-posedness to the null-space may even yield new regularization methods in Hilbert spaces. The resulting surjective operator has a continuous Bartle-Graves inverse, which might become accessible to algorithmic evaluation or approximation in future. To transfer ill-posedness of an operator $A:X\to Y$ between Hilbert spaces to the null-space in Theorem~\ref{th:extension} one may choose $W:=Y$, $C$ the identity mapping in $Y$, and $\tilde{X}:=Y\times X$ with norm
\begin{equation}
\|(w,u)\|_{\tilde{X}}:=\sqrt{\|w\|_W^2+\|x\|_X^2.}
\end{equation}
Then the domain $\tilde{X}$ of the extension $B$ is a Hilbert space again. Remember Remark~\ref{rm:compl}, which states that although each closed subspace of a Hilbert space is complemented there might be quasi-complements which are not complements. Thus, even in Hilbert spaces the null-space $\cN(B)$ may have a quasi-complement $U$ which is not a complement. In this sense $\cN(B)$ can be regarded as `uncomplemented with respect to $U$'.

Next to Mazur-type operators the authors currently are not aware of any other class of operators with uncomplemented null-space. Of particular interest are practically relevant operators with this property.

\bibliography{quasicomplemented}

\end{document}